\documentclass[a4paper]{amsart}
 \usepackage{hyperref}
 \usepackage{amsfonts}
 \usepackage{amsmath}
 \usepackage{graphicx}
 \usepackage{dsfont}
 \usepackage{mathrsfs}
 \usepackage{amsthm}
 \usepackage{url}
 \usepackage{color}
 \usepackage{tikz}
 \usepackage{enumerate}
\newtheorem{theorem}{Theorem}[section]
\newtheorem{corollary}[theorem]{Corollary}
\newtheorem{lemma}[theorem]{Lemma}
\newtheorem{proposition}[theorem]{Proposition}
\newtheorem{remark}[theorem]{Remark}
\newcommand{\Neg}{\mathcal{N}}
\newcommand{\C}{\mathcal{C}}
\renewcommand{\S}{\mathcal{S}}
\newcommand{\SP}{\mathbb{U}}
\newcommand{\LL}{\mathcal{L}}

\newcommand{\F}{\mathcal{F}}
\newcommand{\x}{x}%{{\bf x}}
\newcommand{\y}{y}%{{\bf y}}
\newcommand{\z}{z}%{{\bf z}}
\newcommand{\0}{0}%{{\bf 0}}
\newcommand{\kk}{a}%{{\bf k}}

\newcommand{\aaa}{a}%{{\bf a}}
\newcommand{\R}{\mathbb{R}}%{\mathds{R}}
\newcommand{\Gr}{G_{r}}

\newcommand{\Vr}{V}%{V_{r}}
\newcommand{\tildeVr}{\widetilde{V}}
\newcommand{\Wr}{W}

\newcommand{\rl}{(r-\LL)}
\newcommand{\hit}[1]{h_{#1}}%{\mathpzc{h}_{#1}}
\newcommand{\indicator}[1]{\mathbf{1}_{#1}}
\newcommand{\g}{{g}}

\newcommand{\er}[1]{e^{-r{#1}}}

\newcommand{\states}{\mathbb{R}^d}%{\mathcal{E}}
\newcommand{\Ex}[2]{\E_{#1}\left(#2\right)}

\def\P{\operatorname{\bf P}}
\def\Pro{\operatorname{\bf P}}
\def\E{\operatorname{\bf E}}
\def\conv{\operatorname{\rm conv}}

\title{On optimal stopping of multidimensional diffusions}
\author{S\"oren Christensen}
\address{S\"oren Christensen, Department of Mathematics, SPST, University of Hamburg, Bundesstra\ss e 55, D-20146 Hamburg, Germany}
\email{soeren.christensen@uni-hamburg.de}

\author{Fabi\'an Crocce}
\address{Fabi\'an Crocce, IMERL, Facultad de Ingenier\'ia, Universidad de la Rep\'ublica, Julio Herrera y Reissig 565, 11200, Montevideo, Uruguay}
\email{fcrocce@fing.edu.uy}

\author{Ernesto Mordecki}
\address{Ernesto Mordecki, Centro de Matematica, Facultad de Ciencias, Igua 4225, 11400, Montevideo, Uruguay}
\email{mordecki@cmat.edu.uy} 

\author{
Paavo Salminen}
\address{Paavo Salminen, \AA bo Akademi University, Faculty of Science
  and Engineering, FIN-20500 \AA bo, Finland}
    \email{phsalmin@abo.fi}



\begin{document}

%%%%%%%%%%%%%%%%%%%%%%%%%%%%%%%%%%%%%%%%%
 \begin{abstract}
This paper develops an approach for solving perpetual discounted optimal stopping problems
for multidimensional diffusions, with special emphasis on the $d$-dimensional Wiener process.
%%%%%%%%%%
We first obtain some verification theorems for diffusions, based on the Green kernel representation
of the value function associated with the problem.
Specializing to the multidimensional Wiener process,
we apply the Martin boundary theory to obtain a set of tractable integral equations involving only harmonic functions
that characterize the stopping region of the problem.
%It turns out that these integral equations have many advantages over alternative equations.
These equations allow to formulate a discretization scheme to obtain an approximate solution.
The approach is illustrated through the optimal stopping problem of a $d$-dimensional Wiener process
with a positive definite quadratic form reward function.
 \end{abstract}
 \maketitle

\section{Introduction and preliminaries}
%%%%%%%%%%%%%%%%%%%%%%%%
%%%%%%%%%%%%%%%%%%%%%%%%
%%%%%% On the optimal stopping problem
%%%%%%%%%%%%%%%%%%%%%%%%
%%%%%%%%%%%%%%%%%%%%%%%%
The theory of optimal stopping in continuous time is a well established and active area of research in stochastic processes.
Initiated in the field of sequential sampling in statistics, it receives a second important impulse with the emergence of mathematical finance and the problem of pricing American type contracts.
For the interested reader we refer to the monogrhaphs
by Shiryaev \cite{shiryaev:2008} and Peskir and Shiryaev \cite{peskir-shiryaev},
which seem to correspond respectively to these two periods of the topic.
It can be observed that, although the general theory of optimal stopping is developed for processes with general
state space, very few explicit solutions seem to be known in concrete examples for underlying multidimensional Markov processes,
which constitute the subject of the present paper.
There is also a relevant distinction between finite horizon and perpetual problems:
in the present work we focus our attention in perpetual problems.

When considering the dimensionality of an optimal stopping problem,
the first question to be answered is whether the problem has a representation in terms of
a one dimensional problem. Important cases of this situation are for instance the optimal stopping
of Bessel processes with integer index, that corresponds to the stopping of functions of the modulus of
multidimensional Brownian motion \cite{dubins-et-al}, or the classical one-asset-for-other problem
of the right to exchange one asset by another introduced by Margrabe \cite{margrabe}
whose American counterpart can also be solved by a change of measure \cite{gerber-shiu,fajardo-mordecki}
with the help of Girsanov's Theorem.
Another relevant example of this type of problem is the Russian options introduced by Shepp and Shiryaev in
\cite{shepp-shiryaev:1993},
that is in principle a two dimensional problem that can be reduced to a one dimensional problem after a convenient change of measure,
as exposed in \cite{shepp-shiryaev:1994},
a similar situation is the one considered for integral options in \cite{kramkov-mordecki:1994}.

%%%%%%%%%%%%%%%%%%%%%%%%
%%%%%%%%%Our framework
%%%%%%%%%%%%%%%%%%%%%%%%

In the present work, according to what has been observed above,
we are interested in the stopping of a multidimensional process when there is no available reduction.
This situation has been by far less explored in the literature.
In conclusion and in short, in this paper we concentrate on multidimensional perpetual discounted
optimal stopping problems for multidimensional diffusions, with special emphasis on the $d$-dimensional Wiener process.

%%%%%%%%%%%%%%%%%%%%%%%%
%%%%%%%%  References for our problem
%%%%%%%%%%%%%%%%%%%%%%%%

Up to our knowledge,
there are a few references in this topic.
Hu and {\O}ksendal \cite{hu-oksendal} consider an investment problem originated in a portfolio of several correlated geometric Brownian motions.
Unfortunately, it turned out that the explicit solution given there holds only in trivial cases,
see Christensen and Irle \cite{christensen-irle} and Nishide and Rogers \cite{NR}.
This second reference discusses the more general problem of exchange of baskets, but only bounds for the solution are presented.
Villeneuve \cite{villeneuve} address the multidimensional case for finite time,
studying the non-emptiness of the stopping region and other properties of this set.
Christensen and Irle \cite{christensen-irle} propose a methodology to obtain information about the stopping region
for multidimensional problems based on harmonic functions. Another of the few attempts for infinite time horizon problems documented in the literature is an adhoc fix point method described for the optimal investment problem in \cite{christensen-salminen}. However, also that method seems to be restricted to some special subclasses of problems.
Crocce \cite{crocce} obtains some verification theorems for
multidimensional Markov processes, using the representation approach.
%and Christensen and Salminen \cite{christensen-salminen}
%characterize the boundary of the optimal stopping set as a unique solution of an integral equation arising immediately from the Riesz representation of the value function in two concrete problems.
%These last two works are the departing point of the present paper.

%%%%%%%%%%%%%%%%%%%%%%%%%%%%%%%%%%%%%%%%%
%%%%%%%%%%%%%%%%%%%%%%%%%%%%%%%%%%%%%%%%%

%%%%%%%%%%%%%%%%%%%%%%%%%%%%%%%%%%%%%%%%%

\subsection{Preliminaries and summary of the approach}\label{subsec:prel}

In the present paper, by a \emph{multidimensional diffusion} we understand a homogeneous nonterminating strong Markov process
$X=\{X_t\}_{t\geq 0}$
with continuous trajectories, defined on a filtered probability space
$(\Omega,\mathcal{F},\mathbb{F}=\{\mathcal{F}_t\}_{t\geq 0},\P)$ and taking values in $\states$.
For each $x\in\states$ we denote $\P_x(\cdot)=\P(\cdot\mid X_0=x)$.
We further assume that $X$ satisfies the Feller property (Proposition III.2.4. in \cite{revuz-yor}) and that the filtration
$\mathbb{F}$
is the result of the construction detailed in Proposition III.2.10 in \cite{revuz-yor}.
In particular (see Thm. III.2.17 in \cite{revuz-yor}), the hitting time of any Borel set is a stopping time.
%The filtration  is assumed to be right-continuous and complete.\marginpar{check}
%%implying that the first entry times to open and closed sets are stopping times.
%In addition it is assumed that the mapping $x\mapsto \P_x(F)$ is measurable
%for each $F\in\mathcal{F}$, and $\mathcal{F}_t$ is $\P_x$-complete for all
%$t\geq 0$ and for all $x\in\states$.

%%%%%%%%%%%%%%%%%%%%%%%%
%%%%%%%%%%% The optimal stopping problem
%%%%%%%%%%%%%%%%%%%%%%%%

Given a diffusion, a nonnegative and continuous reward function $g\colon\states\to[0,\infty)$,
and a positive discount rate $r>0$,
we consider a perpetual optimal stopping problem that consists in finding the value function $\Vr\colon\states\to[0,\infty)$
and an optimal stopping time $\tau^*$ such that
\begin{equation}\label{eq:osp}
\Vr(x)= \sup_{\tau}\Ex{x}{\er{\tau} \g(X_{\tau})}=\Ex{x}{\er{\tau^*} \g(X_{\tau^*})},
\end{equation}
where the supremum is taken over the class of all stopping times w.r.t. the filtration $\mathbb{F}$. In case $\tau=\infty$, we set $\er{\tau} \g(X_{\tau}):=\limsup_{t\rightarrow\infty}\er{t} \g(X_{t})$.
To guarantee finiteness of the value function, the following
\emph{growth condition} is often assumed to hold for all $x$
\begin{align}\label{eq:growth}
\E_x \left(\sup_{t\geq0}e^{-rt}g(X_t)\right)<\infty.%\qquad
%\nonumber
\end{align}
This is, however, often difficult to check for specific examples. We discuss more explicit conditions in our particular situation in Subsection \ref{subsec:finite}.
%If we furthermore have
%\begin{align}\label{eq:growth2}
%\lim_{t\to\infty}e^{-rt}g(X_t)=0~~\Pro_x-\mbox{a.s.},
%\end{align}
%we can assume that the payoff is null on the set $\{\tau=\infty\}$.
For general reference on optimal stopping see \cite{shiryaev:2008} and \cite{peskir-shiryaev}.

%%%%%%%%%%%%%%%%%%%%%%%%%%%%%%%%%%%%%%%%%
%%%%%%% On the representation approach in optimal stopping}\label{subsec:green_eq}
%%%%%%%%%%%%%%%%%%%%%%%%
%Inspired by a similar reasoning as in finite time horizon problems, see \cite{peskir-shiryaev},
%we consider the following line of argument to tackle the problem.

We now describe a free boundary approach to a solution and our contribution in this context.
From the general theory of Markovian stopping problems, the optimal stopping time $\tau^*$ -- if such a time exists -- can be characterized as the first entrance time into the stopping set
\[\S=\{x\in\states\colon \Vr(x)=g(x)\}.\]
Denoting by $\LL$ the infinitesimal generator of the process $X$,
the value function is furthermore known to fulfill $(r-\LL)\Vr(x)=0$
for $x$ in the continuation set $\C=\S^c$ and -- if $g$ is nice enough
-- it holds that $(r-\LL)\Vr(x)=(r-\LL)\g(x)$ for $x$ in the interior
of the stopping set $\S$. Therefore, if we can make sure that $V$ is
smooth enough to apply a version of Dynkin's formula also on the
boundary of $\Vr$ (\emph{smooth fit}), we obtain for all $t\geq
0$ and $x\in\states$ that 
\begin{align*}
\Vr(x)&=\Ex{x}{\er{t} \Vr(X_{t})}+\int_0^t\Ex{x}{\er{s} (r-\LL)\Vr(X_{s})}ds\\
&=\Ex{x}{\er{t} \Vr(X_{t})}+\int_0^t\Ex{x}{\er{s} (r-\LL)\g(X_{s})1_{\{X_s\in \S\}}}ds.
\end{align*}
Under suitable growth conditions,
% e.g. \eqref{eq:growth}
%and
%\begin{align*}
%\lim_{t\to\infty}e^{-rt}g(X_t)=0~~\Pro_x-\mbox{a.s.},
%\end{align*}
%it holds that
letting $t\rightarrow \infty$ yields
%also holds for $V$,
%$\Ex{x}{\er{t} \Vr(X_{t})}\rightarrow0$ for $t\rightarrow\infty$ is fulfilled,
%imply
\begin{align}\label{eq:value_repr}
\Vr(x)&=\int_0^\infty\Ex{x}{\er{s} (r-\LL)\g(X_{s})1_{\{X_s\in \S\}}}ds\nonumber\\
&=\int_{\S}G_r(\x,\y)(r-\LL)g(\y)d\y,
\end{align}
where $G_r$ denotes the  $r$-Green kernel and is given by 
%we find it suitable to use the notion of the $r$-Green kernel $G_r$ given by
\begin{align}\label{eq:green_measure}
\Gr(x,H):=\int_0^\infty \er{s} \P_x(X_t\in H)ds.
\end{align}
By applying the same reasoning to $\g$, we obtain -- on the other hand -- that
\[
\g(x)=\int_{\states}G_r(\x,\y)(r-\LL)g(\y)d\y.
\]
This equation is also known as the \emph{inversion formula}, and  states that the composition of the operators $r-\LL$ and $G_r$ results in the identity operator in a convenient domain
(see for instance Prop. 2.1 in \cite{ek}).
In conclusion -- under suitable assumptions on $g$ and $X$ --
we expect that the value function fulfills
	\begin{equation}\label{eq:char_green_all}
0=\Vr(x)-\g(x)=\int_{\S^c}G_r(\x,\y)(r-\LL)g(\y)d\y,\;\x\in \S.
\end{equation}
In particular, as $\S$ is known to be a closed set under standard assumptions, we obtain that
	\begin{equation}\label{eq:char_green}
\int_{\S^c}G_r(\x,\y)(r-\LL)g(\y)d\y=0,\;\x\in \partial\S,
\end{equation}
which can be interpreted as a (highly nonlinear) integral equation for the unknown boundary of the stopping set. In a next step, the idea is to find a suitable class $\SP$ of candidate sets $U$ to characterize the continuation set $U=\S^c$ as the unique solution to equation
	\begin{equation*}
\int_{U}G_r(\x,\y)(r-\LL)g(\y)d\y=0,\;\x\in \partial U,
\end{equation*}
 in the class $\SP$. We call \eqref{eq:char_green} the \emph{Green kernel equation} for the stopping set. 
 The previously described approach is nowadays standard in the solution of optimal stopping problems with a free-boundary approach, see \cite{peskir-shiryaev} for the relevant theory and a whole bunch of examples, mainly with a finite time horizon.

%(see Section \ref{sec:Green_char} for general results in this direction).
From a mathematical point of view,
this is a satisfying solution to the problem.
However,
from a practical point of view,
finding a numerical solution to set equations as \eqref{eq:char_green} does not seem to be standard.
One main problem is that the set of test points $x\in \partial U$ changes for each candidate set
$U$.
Furthermore, in many situations of interest, the Green kernel $\Gr(x,y)$ is not easy to handle as it has singularities, see below for the Wiener process case.

We now specialize to the Wiener process case.
%Our approach, starts with equation \eqref{eq:char_green_all} for the stopping boundary above.
In order to avoid evaluation at each point of the boundary of the (unknown) stopping set, we proceed inspired by the general Martin boundary theory. More precisely, we do not evaluate \eqref{eq:char_green_all} at the boundary of $S$, but take the limits when $\|x\|\rightarrow \infty$ along rays parametrized by a parameter ${\aaa}$ in the sphere with squared radius
$||{\aaa}||^2=2r$.
Then, for each $\aaa$,
there exists a function $h_a$ such that
$$
\frac{G_r(\x,\y)}{G_r(\x,\0)}\to h_\aaa(\y).
$$
It is more concretely proved in Section \ref{sec:Green_char} that $h_\aaa(y)=e^{{\aaa}\cdot\y}$ where $\cdot$ denotes the scalar product. Now, if $\S^c$ is bounded, we obtain from equation \eqref{eq:char_green_all} by taking the limit as above
	\begin{equation}\label{eq:char_exp}
\int_{\S^c}e^{{\aaa}\cdot\y}(r-\LL)g(\y)d\y=0,\text{ ${\aaa}\in\R^d$ such that $||{\aaa}||^2=2r$}.
\end{equation}
The new system \eqref{eq:char_exp} has many advantages over the Green
kernel equation \eqref{eq:char_green}. Firstly, the parameters $\x$ in
\eqref{eq:char_green}  depend directly on $\S$; the parameters ${\aaa}$ in \eqref{eq:char_exp} are the same for each choice of $\S$. Secondly, although the Green kernel $G_r$ can be found explicitly in the case of an underlying Wiener process (see in \eqref{eq12} below), the exponential functions $e^{{\aaa}\cdot\y}$ have the advantage that they do not have singularities, which makes numerical computations much more stable. We call \eqref{eq:char_exp} the \emph{Martin kernel equation} for the stopping set.

In conclusion, the aim of this paper is to develop the theory to use equation \eqref{eq:char_exp} to solve multidimensional optimal stopping problems.
For this, we use a multi-step procedure as follows:
\begin{enumerate}
	\item Establish the Green kernel equation \eqref{eq:char_green} for the optimal stopping set $\S$ of a multidimensional diffusion (see the discussion above and Section \ref{section:2}).
	\item Prove the uniqueness of the solution of the Green kernel equation in a suitable class of candidate sets $\SP$
	(Section \ref{sec:Green_char}).
	\item %Departing from the Green kernel equation \eqref{eq:char_green},
	When
	 the diffusion is the Wiener process, establish the Martin kernel equation \eqref{eq:char_exp} for the optimal stopping set $\S$
	 (Section \ref{sec:OS_Brown}), and its equivalence to the Green kernel equation.
%	\begin{equation*}
%\int_{\S^c}e^{{\aaa}\cdot\y}(r-\LL)g(\y)d\y=0,\text{ ${\aaa}\in\R^d$ such that $||{\aaa}||^2=2r$}
%\end{equation*}
\item Use step 2 to prove that $\S$ can be characterized as the unique solution to the Martin kernel equation in the class $\SP$
(Section \ref{sec:OS_Brown}).
\item Solve the Martin kernel equation \eqref{eq:char_exp} numerically (Section \ref{sec:numerics}), which gives an approximation of the optimal stopping time and -- via \eqref{eq:value_repr} -- also of the value function.
\end{enumerate}

Note that the first two steps (the Green kernel equation-approach) are only needed as auxiliary results in our approach. We indeed prove in Section \ref{sec:OS_Brown} that -- under certain assumptions -- the characterization of the unknown stopping set via the Green kernel characterizations holds if and only if the characterization via the Martin boundary (steps 3 and 4) holds.
The main contribution of this paper is the later characterization. The main tool we use to establish equivalence is the famous Helgason support theorem from the theory of Radon transforms, which seems to be new in this context.

Note that Martin boundary theory has already been used to solve optimal stopping problems before.
This approach was initiated by Salminen in \cite{salminen85}, where the value function
was represented as an integral of the Martin kernel for one dimensional diffusions, see also \cite{mordecki-salminen,crocce,crocce-mordecki} for a point of view from Riesz decomposition.
The theory presented here is, however, different in nature to the previous approaches.

%%%%%%%%%%%%%%%%%%%%%%%%%%%%%%%%%%%%%%%%%
%%%%%%%%%%%%%%%%%%%%%%%%%%%%%%%%%%%%%%%%%
%%%%%%%%%%%%%%%%%%%%%%%%%%%%%%%%%%%%%%%%%
%%%%%%%%%%%%%%%%%%%%%%%%%%%%%%%%%%%%%%%%%

%\newpage

\section{The Green kernel equation for diffusions}\label{section:2}

In this section we formulate a general verification theorem based on the Green kernel for a diffusion $X$.
Denote by $\hit{S}$ the hitting time of a Borel set $S\subseteq \states$, defined by
\begin{equation*}
\hit{S}:=\inf\{t\geq 0\colon X_t\in S\}.
\end{equation*}%\marginpar{check that all used sets are closed or open}
where, as usual, we consider $\hit{S}=\infty$ if the set is empty.
We start by presenting some preliminary results for later reference. To enhance readability, the proofs of this section are given in the appendix.

\begin{lemma} \label{harmonic}
{Let $\tau$ be a stopping time.
Then, the Green measure introduced in \eqref{eq:green_measure} satisfies
\begin{equation*}
\Gr(x,H)=\Ex{x}{\er{\tau}\Gr(X_{\tau},H )}+\Ex{x}{\int_0^\tau\er{s}\indicator{H}(X_s)ds}.
\end{equation*}}
\end{lemma}
The following result shows that in our situation equation
\eqref{eq:char_green} implies equation \eqref{eq:char_green_all}.

\begin{lemma}
\label{multimainlemma}
Let $f\colon \states \to \R$ be a measurable function, $S\subseteq\states$ a Borel set,
and define the functions $\Wr,g\colon\states \to \R$ by
\begin{align}
\label{eq:defW}
\Wr(x)	&:=\int_{S} f(y) \Gr(x,dy),\\
g(x)	&:=\int_{\states} f(y) \Gr(x,dy).\notag
\end{align}
If $g(x)=\Wr(x)$ for all $x\in \partial S$, then, $g(x)=\Wr(x)$ for all $x\in S$.
Furthermore, $\Wr$ satisfies
\begin{equation} \label{eq:WA}
\Wr(x)=\Ex{x}{ e^{-r \hit{S}}W(X_{\hit{S}})}=\Ex{x}{ e^{-r \hit{S}}g(X_{\hit{S}})}.
\end{equation}
\end{lemma}
Using these facts, we can establish a general verification result in terms of Green kernels.
\begin{theorem} \label{theorem:general}
Assume that there exists $f\colon \states \to  \R$ such that
\begin{equation*}
g(x)=\int_{\states} f(y)\Gr(x,dy)\text{ for all $x \in \states$}.
\end{equation*}
Consider a Borel set $S$ such that $f(x)\geq 0$ for all $x \in S$.
Assume that $\tildeVr\colon\states \to \R$ defined by
\begin{equation}
\label{rie}
\tildeVr(x):=\int_{S} f(y)\Gr(x,dy),
\end{equation}
satisfies the following two conditions:
\begin{itemize}
\item[] $\tildeVr(x)\geq g(x)$ for all $x\in \states \setminus S$,
\item[]
$\tildeVr(x)=g(x)$ for all $x \in \partial S$ { or equivalently
\begin{align}\label{eq:boundary0}
\int_{S^c} f(y)\Gr(x,dy)=0\mbox{ for all $x \in \partial S$.}
\end{align}
}
\end{itemize}
Then, %$S$ is an optimal stopping region, 
$\tau^*=\hit{S}$ is an optimal stopping time
and $\tildeVr$ is the value function of the problem \eqref{eq:osp}.
\end{theorem}

In principle, the following corollary provides a practical way of using  \autoref{theorem:general},
based on the notion of extended infinitesimal generator (see \cite[Chap. VII, Sect. 1]{revuz-yor}). Let us, however, mention that the usability of the verification results obtained here is limited in a multidimensional setting as it is not easy to guess a suitable candidate. As outlined above,
%in Subsection \ref{subsec:green_eq}, the
one standard way to establish the Green kernel equation is to establish smooth fit condition to be able to apply a suitable version of Dynkin's lemma. This must, however, be proved on a case-by-case basis, see \cite{christensen-salminen} (or -- in the finite time case -- \cite{peskir-shiryaev}) for examples.

\begin{corollary}
\label{mainmulti}
Assume that $g\colon\states \to \R$ belongs to the domain
of the extended infinitesimal generator associated with the $r$-killed process,
that the growth condition \eqref{eq:growth} holds,
and also that
$$
\Ex{x}{\int_0^{\infty} \er{s}|\rl g(X_s)| ds}<\infty.
$$
Suppose that $S$ is a measurable set such that 
%satisfying 
$\rl g(x)\geq 0$ for all $x \in S$
and define $\tildeVr\colon\states \to \R$ by
\begin{equation*}
\tildeVr(x):=\int_{S} \rl g(y)\Gr(x,dy).
\end{equation*}
Moreover, assume that the following two conditions hold:
\begin{itemize}
\item[] $\tildeVr(x)\geq g(x)$ for all $x\in \states \setminus S$,
\item[] $\tildeVr(x)=g(x)$ for all $x \in \partial S$.
\end{itemize}
Then, $S$ is an optimal stopping region, $\tau^*=\hit{S}$ is an optimal stopping time,
and $\tildeVr$ is the value function of the problem \eqref{eq:osp}.
%then $S$ is the stopping region of the OSP, being $\hit{S}$ the optimal stopping time and $\Vr$ the value function.
\end{corollary}

%%%%%%%%%%%%%%%%%%%%%%%%%%%%%%%%%%%%%%%%%
%%%%%%%%%%%%%%%%%%%%%%%%%%%%%%%%%%%%%%%%%
%%%%%%%%%%%%%%%%%%%%%%%%%%%%%%%%%%%%%%%%%
%%%%%%%%%%%%%%%%%%%%%%%%%%%%%%%%%%%%%%%%%

 %\newpage

 \subsection{Uniqueness of the solution}\label{sec:Green_char}

%From now on, the underlying diffusion is a standard $d$-dimensional Wiener process
%$W=\{W_t\colon t\geq 0\}$, although the results of Theorem 3.1 can be adapted to more general situations as well.

We now characterize the unknown stopping set $\S$ as the unique solution to the Green kernel equations
\eqref{eq:char_green} in a convenient class of sets.
The proof follows the line of argument for corresponding finite time problems as applied in \cite{peskir-shiryaev} for many different examples.
 Here, we give a general result in the setup of Theorem \ref{theorem:general}:
\begin{theorem}\label{thm:uniqueness_green}
Assume that there exists a continuous function $f\colon \states \to
\R$ such that 
%that satisfies
\begin{equation}\label{eq:representation}
g(x)=\int_{\states} f(y)\Gr(x,dy),\text{ for all $x \in \states,$}
\end{equation}
and that the value function $\Vr\colon\states \to \R$ is of the form as in Theorem \ref{theorem:general}, i.e.
\begin{equation*}
\Vr(x)=\int_{\S} f(y)\Gr(x,dy),
\end{equation*}
so that the hitting time of $\S$ is optimal.
Furthermore, let $\SP$ be a class of nonempty closed subsets of
$\states$ containing $\S$ such that
\begin{enumerate}[\rm (a)]
	\item\label{class:a} $f(y)>0$ for all $U\in \SP$ and $y\in U$.
	\item\label{class:b} For all $U,U'\in\SP$ the following implications hold true:
	\begin{enumerate}[\rm (i)]
		\item If $\P_x(\lambda(\{t\leq h_{U^c}:X_t\not\in U'\})=0)=1$ for all $x\in U\cap U'$, then $U\subseteq U'$. Here, $\lambda$ denotes the Lebesgue measure.
		\item If $\P_x(\lambda(\{t\leq h_{U}:X_t\in U'\})=0)=1$ for all $x\in U^c$, then $U'\subseteq U$.
	\end{enumerate}
\end{enumerate}
Then $U=\S$ is the unique set in the class $\SP$ which satisfies 
%the system of equations
\begin{equation}\label{eq:assume}
\int_{U^c} f(y)\Gr(x,dy)=0\mbox{ for all }x\in\partial U.
\end{equation}
\end{theorem}

\begin{proof}
Let $U\in \SP$ such that \eqref{eq:assume} holds and define
%\begin{equation}\label{eq:assume}
%\int_{U^c} f(y)\Gr(x,dy)=0\mbox{ for all }x\in\partial U,
%\end{equation}
%and write
\[
W(x):=\int_Uf(y)\Gr(x,dy).
\]
Note that by \eqref{eq:representation} and \eqref{eq:assume} we have $W(x)=g(x)\mbox{ for all }x\in\partial U$ and by Lemma \ref{multimainlemma} we also have
$W(x)=g(x)\mbox{ for all }x\in U$.
The rest of the proof is divided into three steps.
%\begin{enumerate}
%\item
\vskip1mm\par\noindent
\textit{Step 1.} %\label{unique1}
We first prove that $W(x)\leq \Vr(x)$ for all $x$:\\
For $x\in U$, this holds as $W(x)=g(x)\leq V(x)$ as $V$ majorizes $g$. Now, let $x\in U^c$.
Then, as in \eqref{eq:WA} in Lemma \ref{multimainlemma}, we obtain
\begin{align*}
W(x)%&=\Ex{x}{\er{h_U}\Gr(X_{h_U},f|_U)}+\Ex{x}{\int_0^{h_U}\er{s}1_U(X_s)f(X_s)ds}\\&
=\Ex{x}{\er{h_U}W(X_{h_U})}=\Ex{x}{\er{h_U}g(X_{h_U})}\leq \Vr(x),
\end{align*}
because $\P_x(X_{h_U}\in \partial U)=1$.
% $\P_x$-a.s. we obtain
%\[W(x)=\Ex{x}{\er{h_U}g(X_{h_U})}\leq \Vr(x).\]
\vskip1mm\par\noindent
\textit{Step 2.} %\item
Let $x\in \S\cap U$ and $\gamma:=h_{\S^c}$. Using Lemma \ref{harmonic} and step 1 we have
\begin{align*}
g(x)=W(x)&=\Ex{x}{\er{\gamma}W(X_{\gamma})}+\Ex{x}{\int_0^{\gamma}\er{s}1_U(X_s)f(X_s)ds}\\
&\leq \Ex{x}{\er{\gamma}V(X_{\gamma})}+\Ex{x}{\int_0^{\gamma}\er{s}1_U(X_s)f(X_s)ds}
\end{align*}
and on the other hand
\begin{align*}
g(x)=V(x)&=\Ex{x}{\er{\gamma}V(X_{\gamma})}+\Ex{x}{\int_0^{\gamma}\er{s}1_\S(X_s)f(X_s)ds}\\
&=\Ex{x}{\er{\gamma}V(X_{\gamma})}+\Ex{x}{\int_0^{\gamma}\er{s}f(X_s)ds}.
\end{align*}
Subtracting these formulas yields
%Combining both formulas yields
\begin{align*}
0&\geq \Ex{x}{\int_0^{\gamma}\er{s}f(X_s)ds}-\Ex{x}{\int_0^{\gamma}\er{s}1_U(X_s)f(X_s)ds}\\
&=\Ex{x}{\int_0^{\gamma}\er{s}1_{U^c}(X_s)f(X_s)ds}.
\end{align*}
As $f(X_s)>0$ in the integral above by assumption \eqref{class:a}, we obtain that
\[\P_x(\lambda(\{t\leq \gamma:X_t\in U^c\})=0)=1\]
which implies $\S\subseteq U$ by assumption (\ref{class:b}(i)).
%\item
\vskip1mm\par\noindent
\textit{Step 3.}
Now assume that $x\in \S^c$. Recall that $\tau^*=\inf\{t:X_t\in \S\}$ is the optimal stopping time. Using Lemma \ref{harmonic} and step 1 %\ref{unique1}
again, we obtain
\begin{align*}
V(x)\geq W(x)&=\Ex{x}{\er{\tau^*}W(X_{\tau^*})}+\Ex{x}{\int_0^{\tau^*}\er{s}1_U(X_s)f(X_s)ds}.
\end{align*}
As $\S\subseteq U$ by the previous step,$~X_{\tau_U}\in \partial U$ and $W=g$ in $U$, we deduce
\begin{align*}
V(x)&\geq \Ex{x}{\er{\tau^*}g(X_{\tau^*})}+\Ex{x}{\int_0^{\tau^*}\er{s}1_U(X_s)f(X_s)ds}\\
&=V(x)+\Ex{x}{\int_0^{\tau^*}\er{s}1_U(X_s)f(X_s)ds},
\end{align*}
where the optimality of $\tau^*$ is also used. Again, as $f(X_s)>0$ in the integral above by assumption \eqref{class:a}, we obtain that
\[\P_x(\lambda(\{t\leq \tau^*:X_t\in U\})=0)=1,\]
which yields that $U\subseteq \S$ by assumption (\ref{class:b}(ii)). Together with the previous step, we have $U=\S$ as claimed.
%\end{enumerate}
This concludes the proof.
\end{proof}

%%%%%%%%%%%%%%%%%%%%%%%%%%%%%%%%%%%%%%%%%
%%%%%%%%%%%%%%%%%%%%%%%%%%%%%%%%%%%%%%%%%
%%%%%%%%%%%%%%%%%%%%%%%%%%%%%%%%%%%%%%%%%
%%%%%%%%%%%%%%%%%%%%%%%%%%%%%%%%%%%%%%%%%

{%\newpage

\section{The Martin kernel equation for the Wiener process}\label{sec:OS_Brown}}
%
%
%\subsection{Establishing the Martin kernel equation}
%
 \subsection{Martin boundary for the $d$-dimensional exponentially killed Wiener process}
From now on, the underlying diffusion is a standard $d$-dimensional Wiener process
$W=\{W_t\colon t\geq 0\}$.
Recall that the transition density of $W$ is given by
\begin{equation}
\label{eq11}
p(t;\x,\y)=(2\pi t)^{-d/2}\exp\left(-\frac{||\x-\y||^2}{2t}\right),
\end{equation}
and -- using formula (29) in Erdelyi et al. \cite{erdelyi} p. 146 -- the
Green kernel is given by
\begin{align}
\label{eq12}
\nonumber
G_r(\x,\y)&= \int_0^\infty e^{-rt}p(t;\x,\y)dt \\
&=(2\pi)^{-d/2}2\left(\frac{||\x-\y||^2}{2r}\right)^{(2-d)/4}K_{(2-d)/2}
\left(||\x-\y||\sqrt{2r}\right),
\end{align}
where $r\geq 0$ (strictly positive in case $d=1$ or $2$) and $K_\nu$ is the modified Bessel
function of the second kind as defined in Lebedev \cite{lebedev} p. 109 and called the Macdonald function. Recall (see, e.g.,
Lebedev \cite{lebedev} p. 136) that for all
$\nu\in\R$
\begin{equation}
\label{eq13}
\lim_{u\to\infty}\sqrt{u}\, {\rm e}^u\, K_\nu(u)=\sqrt{\pi/2}.
\end{equation}
\begin{lemma}
\label{lem11}
For $\y$ fixed, $d\geq 2,$ and all $||\x||$ large enough it holds
$$
\frac{G_r(\x,\y)}{G_r(\x,\0)}\leq 2^{(d+3)/2}\,
\exp\left(\sqrt{2r}\,||\y||\right).
$$
\end{lemma}
\begin{proof} Let $\nu:=(d-2)/2$ and use the fact that
 $K_{-\nu}=K_\nu$ to obtain
\begin{equation}
\label{eq14}
\frac{G_r(\x,\y)}{G_r(\x,\0)} =
\left(\frac{||\x||}{||\x-\y||}\right)^\nu
\frac{K_\nu(||\x-\y||\sqrt{2r})}{K_\nu(||\x||\sqrt{2r})}.
\end{equation}
Applying (\ref{eq13}) and the triangle inequality we deduce for $||\x||$ large enough that
\begin{align*}
\frac{G_r(\x,\y)}{G_r(\x,\0)}
&\leq 4\left(\frac{||\x||}{||\x-\y||}\right)^{(2\nu+1)/2}
\exp\left(\sqrt{2r}\,(||\x||-||\x-\y||)\right)\\
&\leq 2^{(2\nu+5)/2} \exp\left(\sqrt{2r}\,||\y||\right)
\end{align*}
as claimed.
\end{proof}
\noindent
In the next proposition the limiting behavior of the ratio
of the Green kernels as $||\x||\to\infty$ along the rays starting
from the origin  is characterized.
The result identifies the Martin boundary of an exponentially killed
$d$-dimensional Wiener process
as the unit $d$-dimensional sphere.
\begin{proposition}
\label{lem12}
Let $(\x^{(n)}=(x^{(n)}_1,\dots,x^{(n)}_d))_{n=1}^\infty$ be a sequence of points
in $\R^d$ such
that $||\x^{(n)}||\to\infty.$ Then 
$$
\lim_{n\to\infty}\frac{G_r(\x^{(n)},\y)}{G_r(\x^{(n)},\0)}
$$
exists for all $\y=(y_1,\dots,y_d)\in\R^d$ if and only if 
$$\displaystyle{\lim_{n\to\infty}} \frac{x^{(n)}_i}{||\x^{(n)}||} =\frac{a_i}{||{\kk}||},\quad
i=1,\dots, d
$$ 
for some ${\kk}=(a_1,\dots, a_d)\in\R^d\setminus\{\0\}$.
When the limit exists it is gven by 
% Then for $\y=(y_1,\dots,y_d)\in\R^d$ %and
%${\kk}=(k_1,\dots, k_d)\in\R^d\setminus\{\0\}$
\begin{equation}
\label{eq141}
\frac{G_r(\x^{(n)},\y)}{G_r(\x^{(n)},\0)}\to
\exp\left(\sqrt{2r}\,\frac{{\kk}\cdot \y}{||{\kk}||}\right).
\end{equation}
%when  $||\x^{(n)}||\to \infty$  so that
%$\displaystyle{\lim_{n\to\infty}} x^{(n)}_i/||\x^{(n)}|| =k_i/||{\kk}||,\,
%i=1,\dots, d.$

%\begin{cases}
%\exp\left(\sqrt{\frac{2r}{1+k^2}}(y_1+ky_2)\right)&\mbox{if }
%{\displaystyle{\lim_{n\to\infty} x^{(n)}_2/x^{(n)}_1 =k}} \mbox{as } x^{(n)}_1\to +\infty ,\\
%\exp\left(-\frac{\sqrt{2r}}{\sqrt{1+k^2}}(1,k)\cdot(y_1,y_2)\right)&\mbox{if }
%{\displaystyle{\lim_{n\to\infty} x^{(n)}_2/x^{(n)}_1 =k}} \mbox{as } x^{(n)}_1\to -\infty
%\end{cases}
%$$

\end{proposition}
\begin{proof}
Using (\ref{eq13}) in (\ref{eq14}) yields
\begin{align*}
\hskip.5cm\lim_{n\to\infty}&\frac{G_r(\x^{(n)},\y)}{G_r(\x^{(n)},\0)}\\
&\hskip.5cm=\lim_{n\to\infty}\left(\frac{||\x^{(n)}||}{||\x^{(n)}-\y||}\right)^{(2\nu+1)/2}
\exp\left(\sqrt{2r}\,(||\x^{(n)}||-||\x^{(n)}-\y||)\right)\\
&\hskip.5cm=\lim_{n\to\infty}\exp\left(\sqrt{2r}\,(||\x^{(n)}||-||\x^{(n)}-\y||)\right),
\end{align*}
since, as is easily seen,
$$
\lim_{n\to\infty} \frac{||\x^{(n)}||}{||\x^{(n)}-\y||}=1.
$$
%^{(2\nu+1)/2}=1.
To conclude the proof, notice that
\begin{align*}
||\x^{(n)}||-||\x^{(n)}-\y||&= \frac{2\, \x^{(n)}\cdot \y
  -||\y||^2}{||\x^{(n)}||+||\x^{(n)}-\y||}
\end{align*}
from which the claimed equivalence and formula (\ref{eq141}) are readily
obtained.

%&\to \frac{{\kk}\cdot \y}{||{\kk}||}
%\end{align*}
%by the assumptions that $||\x^{(n)}||\to\infty$ and $
%\x^{(n)}/||\x^{(n)}||\to {\kk}/||{\kk}||.$ 
\end{proof}

\subsection{Finiteness of the value function}\label{subsec:finite}
As pointed out above, the finiteness of the value function is not easy to check for an example on hand using condition \eqref{eq:growth}. We circumvent the direct use of \eqref{eq:growth} for the optimal stopping problem of a Wiener process by giving an estimate in terms of harmonic functions. By the general Martin boundary theory, it is known that the functions
\[\x\mapsto \exp(\kk\cdot x),\]
where $\|a\|^2=2r$, found in the previous subsection are (minimal) $r$-harmonic functions. However, it is also straightforward to check directly that functions of the form
\begin{align}\label{eq:hmu}
f_\mu:x\mapsto \int_{B_{2r}} \exp(\kk\cdot x)\mu(da)
\end{align}
for a finite measure $\mu$ on $B_{2r}:=\{a:\|a\|^2=2r\}$ are $r$-harmonic, so that in particular, the process $(e^{-rt}f_\mu(W_t))_{t\geq 0}$ is a positive martingale.

\begin{proposition}\label{prop:bound}
Assume that there exists a finite measure $\mu$ on $B_{2r}$ such that the function $g/f_\mu$ is bounded. Then the value function in \eqref{eq:osp} is finite.
\end{proposition}

\begin{proof}
Let $c>0$ be such that $g(x)/f_\mu(x)<c$ for all $x\in\R^d$ and consider for each stopping time $\tau$
\begin{align*}
  \E_x(e^{-r\tau}g(W_\tau))&=\E_x\left(e^{-r\tau}f_\mu(W_\tau)\frac{g(W_\tau)}{f_\mu(W_\tau)}\right)\\
  &\leq c \E_x\left(e^{-r\tau}f_\mu(W_\tau)\right)\\&\leq c\,f_\mu(x)<\infty.
\end{align*}
Taking supremum over $\tau$ proves the claim.
\end{proof}

\begin{corollary}\label{coro:bound}
In the case $d=2$, if $g$ is such that
\[\sqrt{\|x\|}~e^{-\sqrt{2r}\|x\|}~g(x)\rightarrow 0~\mbox{ as }\|x\|\rightarrow\infty,\]
then the value function in \eqref{eq:osp} is finite.
%is bounded, where $I_0$ denotes the modified Bessel function of first kind of index 0.
\end{corollary}

\begin{proof}
Taking the uniform distribution $\mu$ on the sphere in \eqref{eq:hmu} yields
\begin{align*}
  f_\mu(x) &=\int\exp\bigg(\|a\|\,\|x\|\,\cos(\gamma(a,x))\bigg)\,\mu(da)\\
  &=\frac{1}{2\pi}\int_0^{2\pi}\exp\bigg(\sqrt{2r}\,\|x\|\,\cos(\theta)\bigg)d\theta\\
  &=\frac{1}{\pi}\int_0^{\pi}\exp\bigg(\sqrt{2r}\,\|x\|\,\cos(\theta)\bigg)d\theta\\
  &=I_0(\sqrt{2r}\,\|x\|),
\end{align*}
where $\gamma(a,x)$ denotes the angle between $a$ and $x$ and $I_0$ denotes the modified Bessel function of first kind of index 0. For the final equality see \cite{abramowitzstegun70}. For $I_0$ we have the following asymptotic formula
\[I_0(z)\simeq \frac{1}{\sqrt{2\pi z}}e^z\mbox{ as $z\rightarrow\infty$},\]
which -- when combined with Proposition \ref{prop:bound} -- completes the proof.
\end{proof}

%Recall that $I_0(z)\simeq \frac{1}{\sqrt{2\pi z}}e^z$ as $z\rightarrow\infty$.
%x\mapsto \frac{g(x)}{I_0(\sqrt{2r}|x|)}
\subsection{Equivalence of the Green and Martin kernel equations}

The next theorem establishes that the Green kernel equation implies the Martin kernel equation if the continuation set $\C$ is bounded.
\begin{theorem}\label{thm:nessecary}
%In the setting of Section \ref{section:2},
Assume that there exists a bounded measurable set $C$ such that $g$ is $C^2$ in a neighborhood of $C$ and
for all $\x\in S:=C^c$ the Green kernel equation
\begin{equation}\label{eq:zeroG1}
\int_{C}G_r(\x,\y)(r-\LL)g(\y)d\y=0
\end{equation}
holds true.
Then, the Martin kernel equation
\begin{equation}\label{eq:zeroG2}
\int_{C}e^{{\aaa}\cdot\y}(r-\LL)g(\y)d\y=0\text{ for all } {\aaa}\in B_{2r}
%\R^d$ such that $||{\aaa}||^2=2r$}
\end{equation}
holds as well.
\end{theorem}
\begin{proof} Divide by $G_r(\x,0)$ in \eqref{eq:zeroG1} and take limits in $\x$ along rays.
Combining Lemmas \ref{lem11} and \ref{lem12} with the dominated convergence theorem, the result follows.
\end{proof}

%%%%%%%%%%%%%%%%%%%%%%%%%%%%%%%%%%%%%%%%%%%%%%%%%%%%%%%%
%%%%%%%%%%%%%%%%%%%%%%%%%%%%%%%%%%%%%%%%%%%%%%%%%%%%%%%%
%%%%%%%%%%%%%%%%%%%%%%%%%%%%%%%%%%%%%%%%%%%%%%%%%%%%%%%%
%%%%%%%%%%%%%%%%%%%%%%%%%%%%%%%%%%%%%%%%%%%%%%%%%%%%%%%%

%\newpage

Next, we show that -- {under natural assumptions} -- condition \eqref{eq:zeroG2} implies \eqref{eq:zeroG1} also.
We start with an easy lemma that sheds light on the probabilistic connection of $G_r$ and the Martin boundary.

\begin{lemma}
Let $H$ be an affine hyperplane of the form
$H=\{\y\in\R^d:{\aaa}\cdot\y=b\}$ for some ${\aaa}\in B_{2r},$
$\;b\in\R$, and introduce $\hit{H}=\inf\{t\geq0:W_t\in H\}$.
Then, for each $\x\in\R^d$,
\[4r\int_HG_r(\x,\y)\rho(d\y)=\E_\x(e^{-r\hit{H}})=e^{-\|{\aaa}\cdot \x-b\|},\]
where $\rho$ denotes the surface measure on $H$.
\end{lemma}

\begin{proof}
Note that $\hit{H}=\inf\{t\geq0:{\aaa}\cdot W_t=b\}$ and the process $\widetilde{W}:={\aaa}\cdot W$ is a one-dimensional Wiener process with standard deviation $2r$. Therefore, the second identity holds by {\cite[II,1,2.0.1]{borodin-salminen}}. For the first one, note that{ for an independent exponential time $T$ with parameter $r$, we have
\begin{align*}
\int_HG_r(\x,\y)\rho(d\y)&=\frac{1}{r}\int_H\frac{\Pro_{\x}({W}_T\in d\y)}{d\y}\rho(d\y)=\frac{1}{r}\frac{\Pro_{\x}(\aaa{W}_T\in db)}{db}\\
&=\frac{1}{r}\frac{\Pro_{{\aaa}\x}(\tilde{W}_T\in db)}{db}=\frac{1}{4r}e^{-\|\aaa\x-b\|},
\end{align*}
where we used \cite[I, Appendix 1, 1]{borodin-salminen} in the last step. }
\end{proof}

%\begin{proposition}\label{prop:reverse}
\begin{theorem}\label{thm:suff}
Let $f\colon\R^d\to\R$ be continuous
and $C$ be a {bounded measurable set such that $f$ is $C^2$ in a neighborhood of $C$} and assume that
\begin{equation}\label{eq:zero}
\int_{C}e^{{\aaa}\cdot\y}f(\y)d\y=0,\text{ for all } {\aaa}\in B_{2r}.
%$\R^d$ such that $||{\aaa}||^2=2r$}.
\end{equation}
 Then for all $\x\in \R^d\setminus \conv(C)$
\begin{equation}\label{eq:zero2}
\int_{C}G_r(\x,\y)f(\y)d\y=0,
\end{equation}
where $\conv(C)$ denotes the convex closure of the set $C$.\\
If furthermore $S=C^c$ is closed and connected, then \eqref{eq:zero2} holds true for all $x\in S$.
%Assume that there exists a set $\C$ such that
%\begin{equation}\label{eq:zero}
%\int_{\C}e^{ax}(r-\LL)g(x)dx=0,\text{ for all $a\in\R^d$ such that $\|a\|^2=2r$}.
%\end{equation}
%Define $\stopping=\states\setminus\C$. Then, the optimal stopping problem \eqref{eq:osp} has
%value function
%$$
%V(x)=\int_{\stopping}G_r(x-y)(r-\LL)g(y)dy,
%$$
%and
%$$
%\tau=\inf\{t\geq 0\colon X_t\in\stopping\},
%$$
%is an optimal stopping time for the problem.
%\end{proposition}
\end{theorem}

\begin{proof}
We write
\[f_C(x)=\int_{C}G_r(\x,\y)f(\y)d\y\]
for short. Let $H\subseteq \R^d\setminus  C$ be an arbitrary hyperplane not intersecting $C$, where the parameters ${\aaa}$ and $b$ are chosen such that ${\aaa}\cdot\y\leq b$ for all $\y\in C$. Then, using the previous lemma,
\begin{align*}
0=&{\frac{1}{4r}}\int_{C}e^{{\aaa}\cdot\y}f(\y)d\y={\frac{1}{4r}}\int_{C}e^{b}e^{{\aaa}\cdot\y-b}f(\y)d\y=\int_{C}e^b\int_HG_r(\y,\z)\rho(d\z)f(\y)d\y\\
=&e^b\int_H\int_{C}G_r(\y,\z)f(\y)d\y\rho(d\z)
=e^b\int_H\int_{C}G_r(\z,\y)f(\y)d\y\rho(d\z)\\
=&e^b\int_Hf_C(\z)\rho(d\z).
\end{align*}
We obtain that the integral of $f_C$ over all hyperplanes not intersecting $C$ vanishes. Furthermore, $f_C$ is continuous and {it is easily checked from \eqref{eq13} and the boundedness of $C$ that }$\|\x\|^kf_C(x)$ is bounded for each integer $k>0$. Now, Helgason's support theorem (see, e.g., \cite[Corollary 2.8]{helgason}) yields that $f_C(x)=0$ for all $x\in\R^d\setminus \conv(C)$, proving the first claim. 

%\end{proof}
%Adjusting it to the stopping situation, we obtain the following result.
%
%\begin{theorem}\label{thm:suff}
%In the optimal stopping problem \eqref{eq:osp} with $X=W$, assume that $\S=\C^c$ is connected,
%$g$ is $C^2$ in a neighborhood of $\C$,
%and the Martin kernel equation
%\begin{equation*}%\label{eq:zeroG2}
%\int_{\C}e^{{\aaa}\cdot\y}(r-\LL)g(\y)d\y=0,\text{ for all ${\aaa}\in\R^d$ such that $||{\aaa}||^2=2r$}
%\end{equation*}
%holds true.
%Then, the Green kernel equation
%\begin{equation*}%\label{eq:zeroG1}
%\int_{\C}G_r(\x,\y)(r-\LL)g(\y)d\y=0,~x\in S,
%\end{equation*}
%holds as well.
%\end{theorem}

%\begin{proof}
%By Proposition \ref{prop:reverse}, the identity $f_\C(x)=0$ ($f_\C$ as above with $f=(r-\LL)g$) holds for all $\x\in \R^d\setminus \conv(\C)$. 

For the second claim assume now that $S=C^c$ is connected. As $C$ is bounded, the set $ \R^d\setminus \conv(\C)$ contains an interior point.
%As $h_\C$ is $r$-harmonic on $\R^d\setminusC$, the maximum principle yields that $h_\C(x)=0$ on $\R^d\setminusC$.
By connectedness, take $\y\in\R^d\setminus C$ and take an analytical curve $\gamma:[0,1]\rightarrow S$ with $\gamma(0)=\y$, $\gamma(1)=\x$, where $\x\in \R^d\setminus \conv(C)$. As $f_C$ is analytic on $\R^d\setminus C$, so is $f:=f_C \circ \gamma.$ Furthermore, $f(\z)=0$ in a neighborhood of $\z=1$ by the first claim. By the identity theorem for analytic functions, we have $f\equiv 0$, which proves that $f_C(y)=0$.
\end{proof}

Adjusting it to the stopping situation, we have have the following uniqueness result.

\begin{corollary}\label{cor:unique}
%In the setting of Section \ref{section:2}, a
In the optimal stopping problem \eqref{eq:osp} with $X=W$, assume that the reward function $g$ is $C^2$ and let $\SP$ be a class of connected sets $S$ such that $C:=\R^d\setminus S$ is bounded. If $S=\S$ is the unique solution to the Green kernel equation
\begin{equation}\label{eq:zeroC1}
\int_{S^c}G_r(\x,\y)(r-\LL)g(\y)d\y=0,\;\x\in S,
\end{equation}
in the class $\SP$, then $S=\S$ is the unique solution to the Martin kernel equation
\begin{equation}\label{eq:zeroC2}
\int_{S^c}e^{{\aaa}\cdot\y}(r-\LL)g(\y)d\y=0,\text{ for all }
    {\aaa}\in B_{2r}
%\R^d$ such that $||{\aaa}||^2=2r$}
\end{equation}
in the class $\SP$ as well.
\end{corollary}

\begin{proof}
Theorem \ref{thm:nessecary} yields that each solution of \eqref{eq:zeroC1} is also a solution to \eqref{eq:zeroC2}. On the other hand, if $C$ is a solutions to \eqref{eq:zeroC2}, then Theorem \ref{thm:suff} yields that $C$ is also a solution to \eqref{eq:zeroC1}, which proves the uniqueness.
\end{proof}

{

\begin{remark}
The results we obtained so far are applicable for bounded continuation sets as we are interested in this case for the examples considered below. The generalization of the results to unbounded continuation sets is not straightforward. One first reason is that one cannot expect \eqref{eq:zeroG2} to hold true for all vectors in the circle
$\|\aaa\|=2r$, but only for those directions $\aaa$ such that there exists a sequence
$(\x^{(n)}=(x^{(n)}_1,\dots,x^{(n)}_d))_{n=1}^\infty$ of points
in the stopping set such
that $||\x^{(n)}||\to\infty$ and
$\displaystyle{\lim_{n\to\infty}} x^{(n)}_i/||\x^{(n)}|| =\aaa_i/||{\aaa}||,\,
i=1,\dots, d.$
To establish a variant for the reverse implication, the main ingredient needed is a variant of Helgason's support theorem for unbounded sets in the spirit of \cite{banh} under suitable assumptions on the geometry of the continuation set.
\end{remark}

\begin{remark}
Note that choosing a standard Brownian motion without drift in the previous considerations is just a matter of standardization. Indeed, optimal stopping problems for $d$-dimensional possibly correlated Brownian motions with possible drift can be transformed to our setting using a Girsanov transform.
\end{remark}

%\subsection{A general approach for the solution of multidimensional optimal stopping problems}
%Now, an approach for the solution of multidimensional optimal stopping problems with reward functions $g$ and an underlying Brownian motion can be summarized as follows:
%\begin{enumerate}
	%\item Use the approach described in Section \ref{sec:Riesz} to establish
	%\begin{equation*}
%\int_{\S^c}G_r(\x,\y)(r-\LL)g(\y)d\y=0,\;\x\in \S.
%\end{equation*}
	%\item Find a suitable class of candidate sets to characterize the optimal stopping set $S$ as the only solution to this equation as described in Section \ref{sec:Green_char}.
	%\item Use the theory develop in this section to characterize the stopping set $S$ as the unique solution to
	%\begin{equation*}
%\int_{\S^c}e^{{\aaa}\cdot\y}(r-\LL)g(\y)d\y=0,\text{ ${\aaa}\in\R^d$ such that $||{\aaa}||^2=2r$}.
%\end{equation*}
%\item Solve this equation numerically using the method described in Section \ref{sec:numerics} below.
%\end{enumerate}

}

{
\section{Numerical approximation of the stopping set through an example}\label{sec:numerics}

Our main result reduces the solution to an optimal stopping problem \eqref{eq:osp}
to the problem of finding a surface -- the optimal stopping boundary --
that satisfies a set of integral equations,
parametrized in the harmonic functions.
As usual, when implementing
this theoretical system of equations (infinite unknowns for infinite equations) the idea is to discretize
the problem, and produce as many equations as unknowns.
Nevertheless, we have no warranty that this procedure, that gives non-linear equations,
can give a numerically stable solution.
Of course, a previous approximate knowledge of the shape and location of the solution can help with the initial set for an approximation algorithm.
In this direction, a key r\^ole is played by the \emph{negative set} of the problem, that we denote by
$$
\Neg=\{x\in\R^d\colon \rl g(x)\leq 0\}.
$$
As we know that the continuation set contains the negative set, i.e. $\Neg\subseteq\C$,
a natural proposal is to depart from this negative set,
considering the continuation set $\C$ as an enlargement of this set $\Neg$.

To illustrate the previous general considerations, we concentrate on a reward functions
of the form $g(x)=\sum_{i=1}^d\lambda_i x_i^2$, with $\lambda_i>0$.
Due to the symmetry in law of the Wiener process, any problem with a positive definite quadratic form payoff function can be reduced to this form. As scaling does not change the optimal stopping rule, we can furthermore assume w.l.o.g. that $\lambda_1=1$. In the two-dimensional case, we use the parametrization $g(x,y)=x^2+\alpha^2 y^2$ for general $\alpha>0$.

\subsection{The symmetric case}\label{subsec:symmetric}

 Let us discuss first the case $d=2$ with $\alpha=1$. This is a particularly easy case, as the problem can be identified to be one-dimensional due to symmetry. Indeed, $Z=\{Z_t=\sqrt{g(X_t)}\colon t\geq 0\}$ is a Bessel process of order 2, so that the problem is equivalent to the optimal stopping problem for the one-dimensional diffusion $Z$ and reward $z^2$. It turns out that the optimal stopping time for the original problem is given by a first exit time from a circle around $0$. The calculations for finding the radius $R=R_r$ can be carried out similarly to those in \cite[Subsection 2.4.6]{crocce}.
Symmetric quadratic rewards in $\R^d$ with $d\geq 3$ allow for the same treatment.

\subsection{The two-dimensional quadratic reward}\label{subsec:quadr}
Now, we consider $g(x,y)=x^2+\alpha^2 y^2$ for general $\alpha$.
We have
$$
\Neg=\Big\{(x,y)\colon x^2+\alpha^2y^2\leq \beta^2\Big\},
$$
where
$$
\beta=\sqrt{1+\alpha^2\over r}.
$$
As $g$ is a polynomial, Corollary \ref{coro:bound} yields the finiteness of the value function.
{In the following lemma, we collect some properties of the optimal stopping set $\S=\{(x,y)\in\R^2:V(x,y)=g(x,y)\}$:
\begin{lemma}\label{lem:stop_prop}
\begin{enumerate}[\rm (i)]
\item $\S$ is closed.
\item $\Neg\subseteq \C=\R^2\setminus \S$.
\item $\S$ is symmetric w.r.t. the Cartesian axis, i.e. if $(x,y)\in \S$, then \linebreak $((-1)^\ell x,(-1)^ky)\in \S$ for all $\ell,k\in \{0,1\}$.
\item For any $(x,y)\in \S$ there exists a cone with vertex $(x,y)$, which is a subset of $\S$.
%$(\tilde x,\tilde y)\in \S$ for all $\tilde x,~\tilde y$ with $|x|\leq |\tilde x|,~|y|\leq |\tilde y|.$
\item $\C$ is star-shaped with origin as a center, i.e., for any $(x,y)\in \C$ the line segment between the origin and $(x,y)$ is also contained in $\C$.
% $(\lambda_1x,\lambda_2y)\in \C$ for all $\lambda_1,\lambda_2$ with $\lambda_1,\lambda_2\in[-1,1].$
\item $\C$ is bounded.
More precisely, for $\alpha\geq 1$ and $(x,y)$ such that $|x|\geq \alpha^2R,\,|y|\geq R,$ where $R=R_r$
denotes the radius of the optimal stopping circle in the symmetric problem discussed in Subsection \ref{subsec:symmetric}, it holds that $(x,y)\in \S.$

 %for all $(x,y)\in \C$ it holds that $|(x,y)|\leq \alpha^2R,$ where $R=R_r$
%denotes the radius of the optimal stopping circle in the symmetric problem discussed in Subsection \ref{subsec:symmetric} (with the same discounting factor $r$).
\end{enumerate}
\end{lemma}}
{The proof can be found in the appendix.}
{Following the lines of Subsection \ref{subsec:prel} (see \cite[Section 3.3]{christensen-salminen} or in the corresponding finite time problems), it is now not difficult to check that the value function $\Vr$ is indeed of the form
\begin{align*}
\Vr(x,y)=\int_{\S} G_r((x,y),(x',y'))\left(r-\LL\right)g(x',y')dx'dy'.
\end{align*}
Therefore, also the Green kernel equation holds true.}{
To apply the general theory developed above, we need to specify a suitable class $\SP$ of sets containing the stopping set. The previous lemma suggests to take the class of all $U$ such that
\begin{enumerate}[($\SP$1)]
\item $U$ is closed.
\item $\Neg\subseteq U^c$.
\item $U^c$ is bounded.
\item For any $(x,y)\in U$ there exists a cone with vertex $(x,y)$, which is a subset of $U$.
\item $U^c$ is star-shaped with origin as a center.
%\item If $(x,y)\in U$, then $(\tilde x,\tilde y)\in \S$ for all $\tilde x,~\tilde y$ with $|x|\leq |\tilde x|,~|y|\leq |\tilde y|.$
%\item If $(x,y)\in U^c$, then $(\lambda_1x,\lambda_2y)\in U^c$ for all $\lambda_1,\lambda_2$ with $\lambda_1,\lambda_2\in[-1,1].$
\end{enumerate}
}{ Indeed, it is straightforward to check that the set $\SP$ fulfills the assumptions of Theorem \ref{thm:uniqueness_green}. Therefore, we know that $\S$ is the only set in $\SP$ such that for all $(x,y)\in \S$
\begin{align*}
0=\int_{\S^c} G_r((x,y),(x',y'))\left(r-\LL\right)g(x',y')dx'dy'.
\end{align*}
Now, putting pieces together, Corollary \ref{cor:unique} yields
\begin{proposition}
The stopping set $\S$ of the stopping problem for quadratic reward is the unique set in class $\SP$ such that
\begin{equation}\label{eq:char_quadr}
\int_{\S^c}e^{{\aaa}\cdot(x,y)}(r-\LL)g(x,y)dxdy=0,\text{ for all ${\aaa}\in\R^2$ such that $||{\aaa}||^2=2r$},
\end{equation}
where { $(r-\LL)g(x,y)=r(x^2+\alpha^2y^2-\beta^2$). }
\end{proposition}}
In order to implement a numerical procedure to find the set $\S$, we change to affine polar coordinates
$
x=\rho\cos\theta,
\alpha y=\rho\sin\theta,
$
in \eqref{eq:char_quadr}. The negative set becomes
$$
\Neg=\Big\{(\theta,\rho)\colon 0\leq\theta<2\pi, 0\leq\rho\leq
\beta\Big\}.
$$
We parametrize the continuation set as
$$
\C=\Big\{(\theta,\rho)\colon 0\leq\theta<2\pi, 0\leq\rho\leq \rho(\theta)
\Big\}
$$
where $\rho(\theta)\geq \beta$ is an unknown curve.
Equation
\eqref{eq:char_quadr} is then
\begin{equation*}%\label{eq:phi2}
\int_0^{2\pi}\int_0^{\rho(\theta)}e^{\sqrt{2r}(\rho \cos(\theta) \cos\phi+(\rho/\alpha) \sin(\theta) \sin\phi)}(r-\LL)g(\rho,\theta)\frac1{\alpha}\rho\ d\rho\ d\theta=0,
\end{equation*}
that, after some rewriting, gives
\begin{equation}\label{eq:integrals}
\int_0^{2\pi}F_2(\rho(\theta),\gamma(\theta,\phi))d\theta=
\int_0^{2\pi}F_1(\gamma(\theta,\phi))d\theta,\ 0\leq\phi\leq2\pi.
\end{equation}
%\begin{multline*}%\label{eq:phi2}
%\int_0^{2\pi}\int_\beta^{\rho(\theta)}e^{\sqrt{2r}(\cos(\theta) \cos\phi+(1/\alpha) \sin(\theta) \sin\phi)\rho}
%(\rho^2-\beta^2)\rho\ d\rho\ d\theta\\
%=
%\int_0^{2\pi}\int_0^\beta e^{\sqrt{2r}(\cos(\theta) \cos\phi+(1/\alpha) \sin(\theta) \sin\phi)\rho}
%(\rho^2-\beta^2)\rho\ d\rho\ d\theta.
%\end{multline*}
%As
%\begin{equation*}
%\int e^{\gamma{s}}({s}^3-\beta^2{s})d{s}%\\
%=
%\frac{{\left(1 - \gamma {s}\right)} \beta^{2} e^{\gamma {s}}}{\gamma^{2}}
%+ \frac{{\left(\gamma^{3} {s}^{3} - 3 \, \gamma^{2} {s}^{2} + 6 \, \gamma {s} - 6\right)} e^{\gamma {s}}}{\gamma^{4}},
%\end{equation*}
%we have
where
\begin{align*}
\gamma(\theta,\phi)&=\sqrt{2r}(\cos(\theta) \cos\phi+(1/\alpha) \sin(\theta) \sin\phi),\\
F_1(\gamma)&%=\int_0^\beta e^{\gamma{s}}(s^3-\beta^2s)d{s}
=
\frac{-2 \, {\left(\gamma^{2} \beta^{2} - 3 \, \gamma \beta + 3\right)} e^{\gamma \beta}}{\gamma^{4}} + \frac{\gamma^{2} \beta^{2} - 6}{\gamma^{4}},
\\
F_2(\rho,\gamma)&%=\int_\beta^{\rho} e^{\gamma{s}}({s}^3-\beta^2{s})d{s}
=\frac{2 \, {\left(\beta^{2} \gamma^{2} - 3 \, \beta \gamma + 3\right)} e^{\beta \gamma}}{\gamma^{4}} \\
&\quad\quad
+ \frac{{\left({\left(\beta^{2} \rho - \rho^{3}\right)} \gamma^{3} - {\left(\beta^{2} - 3 \, \rho^{2}\right)} \gamma^{2} - 6 \, \gamma \rho + 6\right)} e^{\gamma \rho}}{\gamma^{4}}.
\end{align*}
Then, given a positive integer $N$, we discretize \eqref{eq:integrals}, denoting $\theta_i=i/(2\pi N)\ (0\leq i < N)$ and,
in order to determine the unknowns $\rho_i=\rho(\theta_i)$ we impose the same number of conditions,
for values $\phi_j=j/(2\pi N)\ (0\leq j< N)$. We then have the discretized problem
$$
\sum_{i=0}^{N-1}F_2(\rho_i,\gamma(\theta_i,\phi_j))=
\sum_{i=0}^{N-1}F_1(\gamma(\theta_i,\phi_j)),\quad 0\leq i,j< N.
$$
The vector $\rho_i$ is found by a gradient algorithm. In Figure \ref{figure:1} we plot the solutions found for $\alpha=2$ and $r=1$ and $r=0.3$, resp., considering $N=64$. The solution for $\alpha=1/2$ is, as expected, just a 90 degrees rotation of the one with $\alpha = 2$. The solution we found for $\alpha = 1/3$ and $r=1$ is presented in Figure \ref{figure:alphalessone}.
% The values $\rho_0 \ldots \rho_{16}$ for $r=1$ are
%$2.619$, $2.622$, $2.633$, $2.651$, $2.677$, $2.709$, $2.747$, $2.792$, $2.840$, $2.890$, $2.939$, $2.986$, $3.026$, $3.059$, $3.083$, $3.097$ and $3.103$ respectively. In the case $r=0.3$ they are
%$4.784$, $4.790$,  $4.810$,  $4.843$,  $4.889$, $4.947$,  $5.017$,  $5.096$,  $5.182$,  $5.271$, $5.361$,  $5.446$,  $5.521$,  $5.585$,  $5.633$, $5.662$ and $5.672$. The remainder values to complete the vector $\rho$ in both examples are omitted since the solutions present a central symmetry.
%\todo{{@ Ernesto/Fabian: Please, include the latest figures and adjust the captions}}
\begin{center}
\begin{figure}
\includegraphics[scale=0.3]{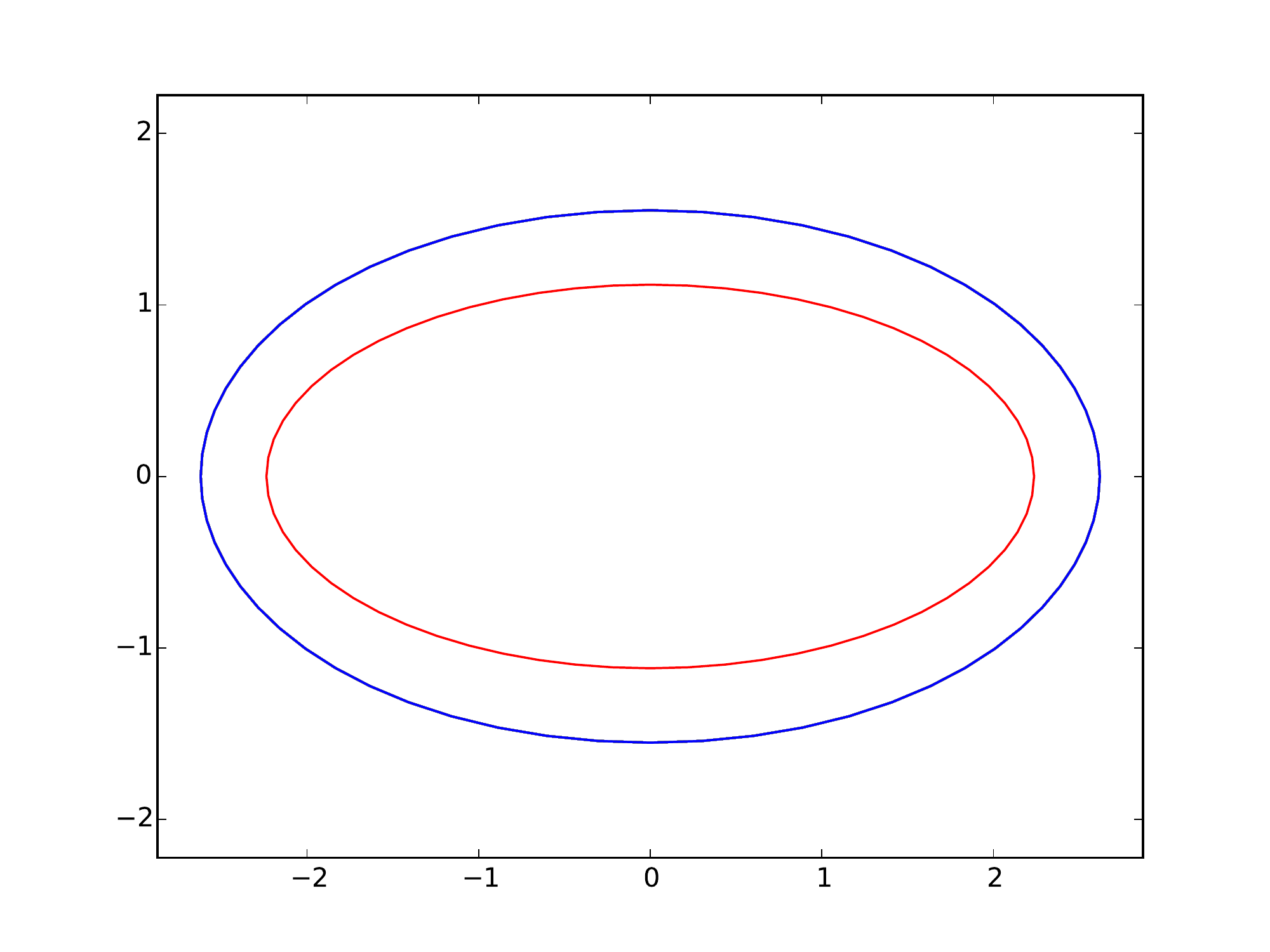}
\includegraphics[scale=0.3]{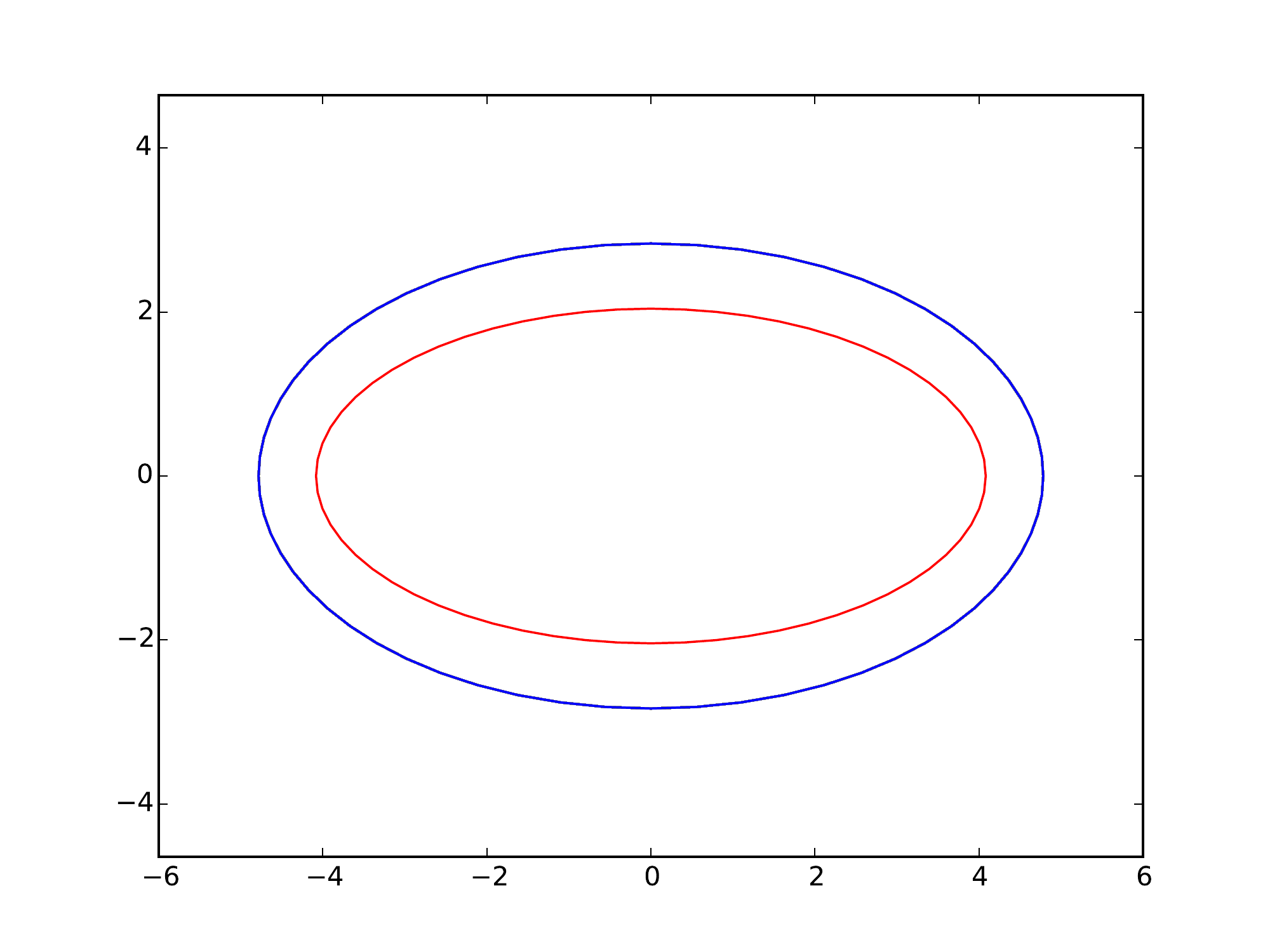}
\caption{Example in \ref{subsec:quadr} with $\alpha = 2$ and with $r=1$ (left) and $r=0.3$ (right): The bounded region delimited in blue is the optimal continuation region. In the region delimited in red $(r-\LL)g$ is negative.}
\label{figure:1}
\end{figure}
\end{center}

\begin{center}
\begin{figure}
\begin{center}
\includegraphics[scale=0.3]{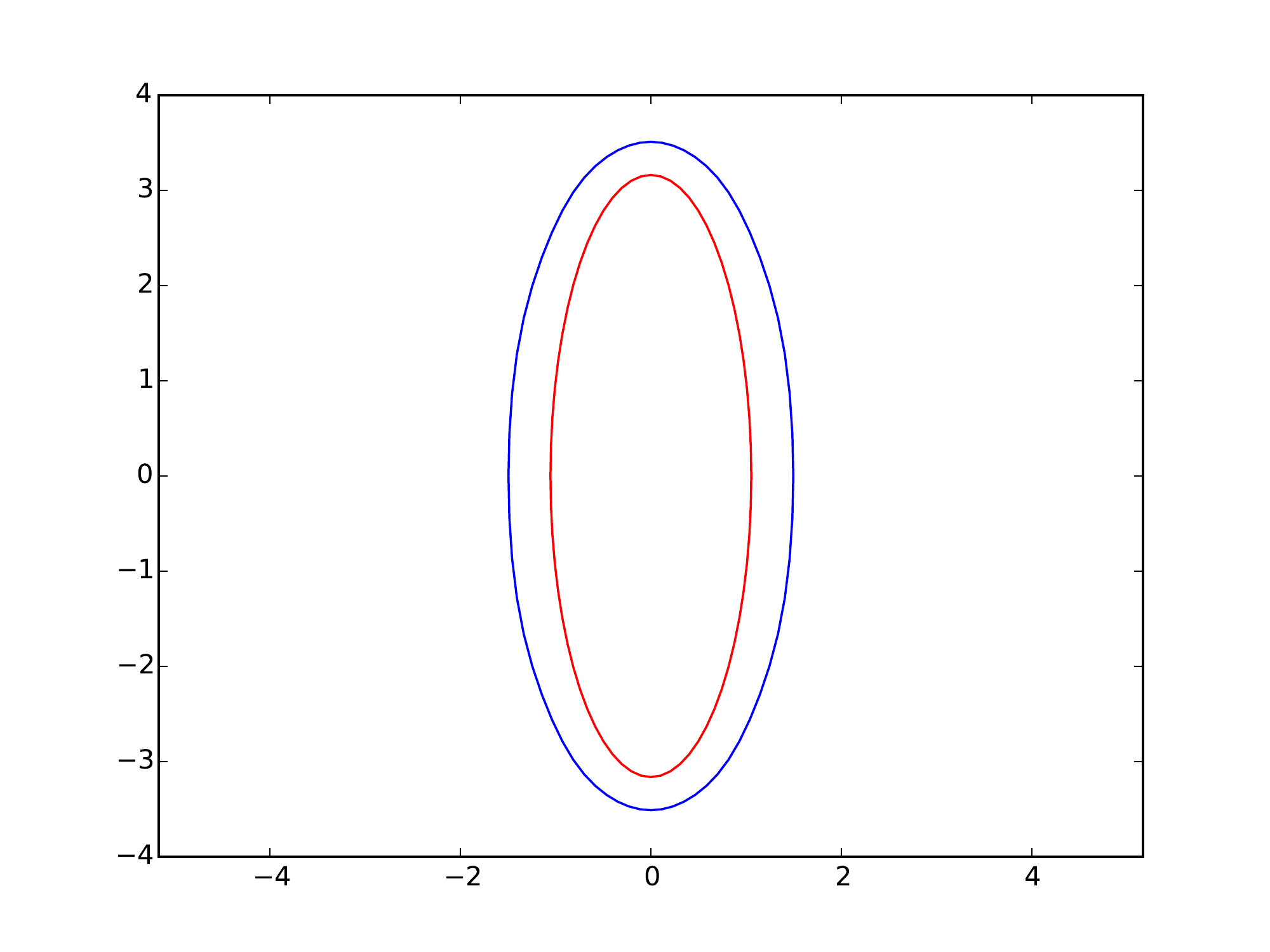}
\caption{Example in \ref{subsec:quadr} with $\alpha = 1/3$ and $r=1$ : The bounded region delimited in blue is the optimal continuation region. In the region delimited in red $(r-\LL)g$ is negative.}
\label{figure:alphalessone}
\end{center}
\end{figure}
\end{center}
%\todo{Could you say more about the numerical solution? Wasn't there also another approach? What is about the one in Section \ref{section:algorithm}?}

\subsection{The $d$-dimensional quadratic reward}
For quadratic reward functions $g$ for general $d$-dimensional Wiener process,
a similar analysis can be carried out. The adaptation to, say, $d=3$ is not hard but the computational resources have to be handled more carefully.
%effort already starts to be a problem.
%The first example we considered, to check the numerical method, was with payoff 
The quadratic reward considered here has the advantage that numerical methods can be checked for the symmetric reward $x^2+y^2+z^2$. 
%The result we obtained is consistent with 
For this problem, the analytic solution is known to be a sphere of radius $w/\sqrt{2r}$ where $w$ satisfies $\tanh(w)=w/3$, analogous to the two-dimensional case discussed in Subsection \ref{subsec:symmetric}.
Figures \ref{fig:3d} show the solution obtained for the payoff $g(x,y,z)=x^2+4y^2+9z^2$, with $r=1$,
and exhibit the corresponding negative set $\Neg$ and continuation region $\C$.
The negative set in this case is
$$
\Neg=\{(x,y,z)\colon\rl g\leq 0\}=\{(x,y,z)\colon x^2+4y^4+9z^2\leq 14\}.
$$

\begin{center}
\begin{figure}
\begin{tabular}{cc}
\includegraphics[scale=0.3]{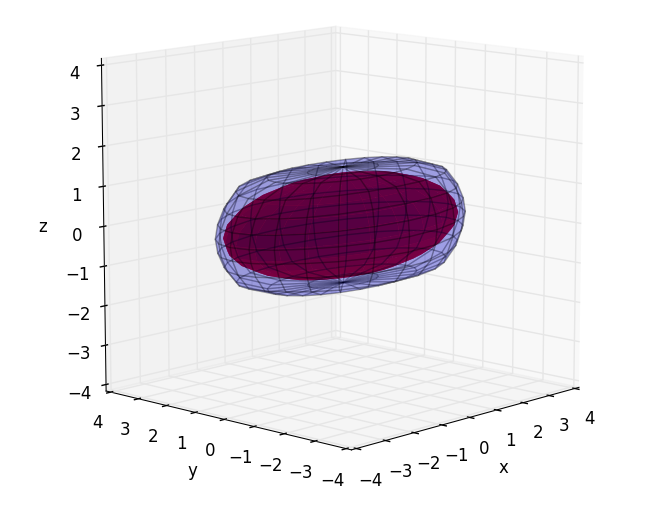} &
\includegraphics[scale=0.3]{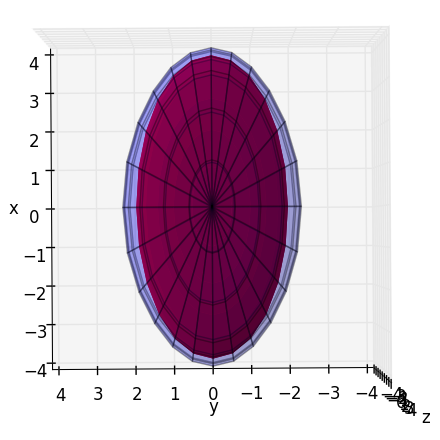}\\
\includegraphics[scale=0.3]{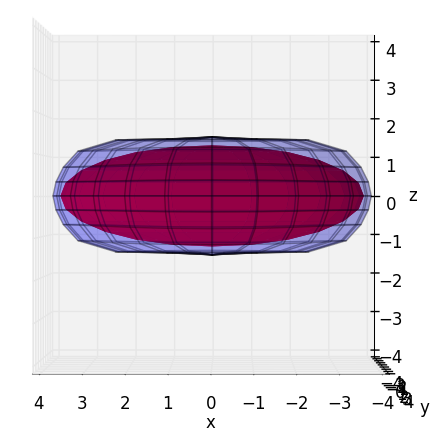}&
\includegraphics[scale=0.3]{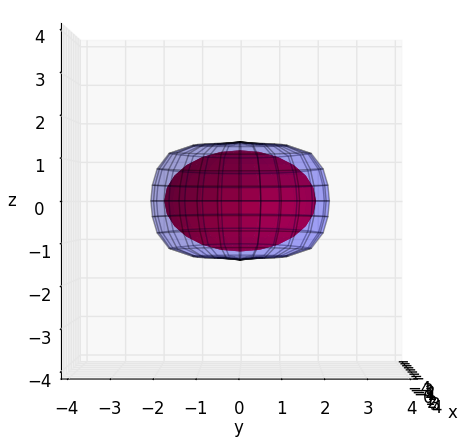}

\end{tabular}

%\hskip16mm\includegraphics[scale=0.4]{3d-secondPlot-august16.png}
\caption{The negative set $\Neg$ (in dark) is inside the continuation set $\C$.}\label{fig:3d}
\end{figure}
\end{center}

%3d-secondPlot-august.16

%%%%%%%%%%%%%%%%%%%%%%%%%%%%%%%%%%%%%%%%%%%%%%%%%%%%%
%%%%%%%%%%%%%%%%%%%%%%%%%%%%%%%%%%%%%%%%%%%%%%%%%%%%%
%%%%%%%%%%%%%%%%%%%%%%%%%%%%%%%%%%%%%%%%%%%%%%%%%%%%%

%\newpage
\begin{appendix}

{\section{Proofs of Lemma \ref{harmonic}, Lemma \ref{multimainlemma}, \autoref{theorem:general}, Corollary \ref{mainmulti}, and
    Lemma \ref{lem:stop_prop}}
}
\begin{proof}[Proof of Lemma \ref{harmonic}]

	%For $x\in S$ the assertion is clearly valid, since $\hit{S}\equiv 0$. Let us consider $x\in \states\setminus S$.
	Let $\tau$ be a stopping time.
	By the definition of $\Gr$ we obtain
	\begin{align*}
		\Gr(x,H)&=\Ex{x}{\int_0^{\infty} \er{t}\indicator{H}(X_t)dt}\\
		%&= \Ex{x}{\int_0^{\infty} \er{t}\indicator{H}(X_t)dt\ \indicator{\{\tau<\infty\}} } \\
		&=\Ex{x}{\int_0^{\tau} \er{t}\indicator{H}(X_t)dt }
		  +\Ex{x}{\int_{\tau}^{\infty} \er{t}\indicator{H}(X_t)dt },
	\end{align*}
	%where the second equality holds, because if $\tau$ is infinite, then $X_t$ does not hit $S$, therefore, $\indicator{H}(X_t)=0$  for all $t$. In the third equality we simply split the integral in two parts.
%Note that the first term on the right-hand side of the previous equality vanishes, since $\indicator{H}(X_t)$ is $0$ when $X_t$ is out of $S$ by our previous argument.
It remains to prove that the second term is equal to $\Ex{x}{\er{\tau} \Gr(X_{\tau},H )}$. The following chain of easily verified equalities completes the proof:
	\begin{align*}
		\Ex{x}{\int_{\tau}^{\infty} \er{t}\indicator{H}(X_t)dt}
		&=\Ex{x}{e^{-r \tau}\int_{0}^{\infty} \er{t}\indicator{H}(X_{t+\tau})dt }\\
		&=\Ex{x}{\er{\tau} \Ex{x}{ \int_{0}^{\infty} \er{t}\indicator{H}(X_{t+\tau})dt \big| \F_{\tau}}}\\
		&=\Ex{x}{\er{\tau} \Ex{X_{\tau}}{\int_{0}^{\infty} \er{t}\indicator{H}(X_{t})dt}}\\
		&= \Ex{x}{\er{\tau} \Gr(X_{\tau},H)}.
	\end{align*}
%the first equality is a change of variable; in the second one, we take the conditional expectation into the expected value, and consider that $\tau$ is measurable with respect to $\F_{\tau}$; the third equality is a consequence of the strong Markov property; while in the last one, we use the definition of $\Gr$.
\end{proof}

\begin{proof}[Proof of Lemma \ref{multimainlemma}]
First observe that from the fact that $g(x)=\Wr(x)$ for all $x\in \partial S$ it follows that
\begin{equation}
\label{eq:auxborde}
g(x)-W(x)=\int_{S^c} f(y)\Gr(x,dy)=0.
\end{equation}
We can now move on to prove
 that $g(x)=\Wr(x)$ for $x\in S$. To verify this, we notice that for
 all $x$
 %must prove that
\begin{align*}
g(x)-W(x)&=\int_{S^c} f(y) \Gr(x,dy).
\end{align*}
%and note that we must prove that the second term on the right-hand side vanishes.
From Lemma \ref{harmonic} we deduce that for $x\in S$
\begin{equation*}
 \int_{S^c} f(y) \Gr(x,dy) = \Ex{x}{e^{-r \hit{S^c}} \int_{S^c} f(y)  \Gr(X_{\hit{S^c}},dy)}.
\end{equation*}
% { Since $X_{\hit{S^c}}\in \partial S$ (by the continuity of
%   the sample paths) it follows from \eqref{eq:auxborde} that}.
   Bearing \eqref{eq:auxborde} in mind, the integral inside the expected value above on the right-hand side vanishes, as $X_{\hit{S^c}} \in \partial{S}$ by the continuity of the sample paths, thus completing the first part of the proof.
The validity of \eqref{eq:WA} is a direct consequence of Lemma \ref{harmonic}.
\end{proof}

\begin{proof}[Proof of \autoref{theorem:general}]
%as in the proof for the one-dimensional case
%We prove that $\tildeVr$ is the minimal $r$-excessive function that dominates the reward function $g$. By Dynkin's characterization, this implies that $\Vr$ is the optimal expected reward.

Applying Lemma \ref{multimainlemma} with $\Wr:=\tildeVr$, we get
$\tildeVr(x)=g(x)$ for $x\in S$, which in addition to the hypothesis
$\tildeVr(x)\geq g(x)$ for all $x\in S^c$, yields that $\tildeVr$ is a
majorant of the reward. By the definition of $\tildeVr$, see
(\ref{rie}), and taking into account that $f$ is a non-negative
function in $S$, we furthermore deduce evoking the Ries representation
that $\tildeVr$ is an $r$-excessive function. By the general theory we therefore have that $\tildeVr \geq \Vr.$
Furthermore by Lemma \ref{multimainlemma}, we have
%\begin{equation*}
% \int_{S^c} f(y) \Gr(x,dy) = \Ex{x}{e^{-r \hit{S^c}} \int_{S^c} f(y)  \Gr(X_{\hit{S^c}},dy)}.
%\end{equation*}
%By \autoref{multimainlemma} we get
\begin{align*}
\tildeVr(x)&=\Ex{x}{e^{-r \hit{S}}g(X_{\hit{S}})}
%\\&
\leq \sup_\tau \Ex{x}{e^{-r \tau}g(X_\tau)},
\end{align*}
and conclude that the desired equality holds.
\end{proof}

\begin{proof}[Proof of Corollary \ref{mainmulti}]
First observe that the function $g$ satisfies the conditions for Dynkin's characterization
(see 1.2.1 in \cite{crocce}).
Since $g$ is in the domain of the extended infinitesimal generator associated with the killed process and the additional hypotheses we made about $g$, we know that %\eqref{eq:invExtended} holds, that is
\begin{equation*}
g(x)=\int_{\states}\rl g(y) \Gr(x,dy)\qquad (x\in \states).
\end{equation*}
Therefore, we are in conditions to apply the previous theorem with $f(y):=\rl g(y)$ to complete the proof.
\end{proof}

{\begin{proof}[Proof of Lemma \ref{lem:stop_prop}]
(i) holds as the the Wiener process is a strong Feller standard Markov
    process and therefore the $r$-excessive functions are lower
    semicontinuous,  see \cite[p.77]{blumenthalgetoor68}.
(ii) hold by the general theory of optimal stopping.\\
For (iii) note that both the distribution of $X$ and the reward
function are symmetric with respect to the Cartesian axes. Hence, so is the value function and therefore also the optimal stopping set. \\
To prove (iv) and (v), first note that by (iii),
it is enough to prove the claims for points in the first quadrant. %under the assumption $x,y,\tilde x,\tilde y ,\lambda_1,\lambda_2\geq 0$. 
We now use the explicit starting state-dependence of $X$. For all stopping times $\tau$ and all $(x,y)$ it holds that
\begin{align*}
q_\tau(x,y):=&\E\left(e^{-r\tau}g(x+W^{(1)}_\tau,y+W^{(2)}_\tau)\right)-g(x,y)\\
=&(x^2+\alpha^2y^2)\big(\E(e^{-r\tau})-1)+2x\E(e^{-r\tau}W^{(1)}_\tau\big)\\&
+2\alpha^2y\E(e^{-r\tau}W^{(2)}_\tau)+\E(e^{-r\tau}({W^{(1)}_\tau}^2+\alpha^2{W^{(2)}_\tau}^2))\\
=&-a_\tau\left((x+d_{1,\tau})^2+\alpha^2(y+d_{2,\tau})^2-c_\tau\right)
\end{align*}
for some constants $a_\tau,\,c_\tau\geq 0,\,d_{1,\tau},\;d_{2,\tau}\in\R$, which depend on $\tau$. Moreover, note that $q_\tau(0,0)\geq0$ for all $\tau$. Using this notation, $(x,y)$ is in the optimal stopping set iff $q_\tau(x,y)\leq 0$ for all $\tau$. Taking complements, this means that $\C$ is a union of elliptic discs $q_\tau(x,y)>0$ containing $(0,0)$, whose axes are parallel to the Cartesian axes. \\
Hence, a point $(x,y),$ $x,y\geq 0$, is in the stopping set $\S$ if $(x,y)$ is outside of all ellipses, i.e. $q_\tau(x,y)\leq0$ for all $\tau$. But then so are all points in a suitable cone with vertex $(x,y)$, which proves (iv).\\
If, on the other hand, $(x,y)\in \C$, then there exists $\tau$ such that $(x,y)$ is inside the ellipse, i.e. $q_\tau(x,y)> 0$. Since, furthermore, the origin is also inside this ellipse, $(v)$ clearly holds.\\
Finally, consider claim (vi) and assume without loss of generality that $\alpha\geq 1$. By the symmetry property (iii), it is enough to establish the boundedness of $\C$ in the first quadrant. First, recall that the continuation set for the symmetric problem $\alpha=1$ is a circular disc, see Subsection \ref{subsec:symmetric} . Therefore, for $y>0$ large enough, we know that for all stopping times $\tau$
\begin{align*}
\E_{(0,y)}(e^{-r\tau}g(X_\tau))\leq \alpha^2\E_{(0,y)}(e^{-r\tau}({W^{(1)}_\tau}^2+{W^{(2)}_\tau}^2))\leq \alpha^2(0^2+y^2)=g(0,y),
\end{align*}
so that $(0,y)\in \S$. Next, we establish that also $(\alpha^2 y,0)\in \S$: For all stopping times $\tau$, it holds
\begin{align*}
q_\tau(0,y)&=\alpha^2(\E e^{-r\tau}-1)y^2+2\alpha^2\E(e^{-r\tau}W^{(1)}_\tau)y+d_\tau\\
q_\tau(\alpha^2y,0)&=\alpha^4(\E e^{-r\tau}-1)y^2+2\alpha^2\E(e^{-r\tau}W^{(2)}_\tau)y+d_\tau.
\end{align*}
Now, fix any stopping time $\tau$ and denote by $\tilde\tau$ a corresponding stopping time such that $(\tau,W^{(1)}_\tau,W^{(2)}_\tau)$ has the same law as $({\tilde\tau},W^{(2)}_{\tilde\tau},W^{(1)}_{\tilde\tau})$. As $(0,y)\in \S$, it holds that $q_{\tilde\tau}(0,y)\leq 0$ and as $\alpha\geq 1$ this also implies that
\begin{align*}
  q_\tau(\alpha^2y,0) &\leq \alpha^2(\E e^{-r\tau}-1)y^2+2\alpha^2\E(e^{-r\tau}W^{(2)}_\tau)y+d_\tau \\
  &= \alpha^2(\E e^{-r\tilde\tau}-1)y^2+2\alpha^2\E(e^{-r\tilde\tau}W^{(1)}_{\tilde\tau})y+d_{\tilde\tau}\\
  &= q_{\tilde\tau}(0,y)\leq 0,
\end{align*}
i.e. $(\alpha^2 y,0)\in \S$ for $y$ large enough.
Next, we consider an arbitrary $(x,y)$. Then
\begin{align*}
  q_\tau(x,y) &= q_\tau(x,0)+q_\tau(0,y)-\E(e^{-r\tau}({W^{(1)}_\tau}^2+\alpha^2{W^{(2)}_\tau}^2)) \\
  &\leq  q_\tau(x,0)+q_\tau(0,y),
\end{align*}
so for $x,y$ large enough, it holds that $q_\tau(x,y)\leq 0$ for all $\tau$, i.e. $(x,y)\in S$, implying that $\C$ is bounded.

%Hence, we have found points in the stopping set on both axes. Now the boundedness follows from (iv).
\end{proof}}
\end{appendix}
%%%%%%%%%%%%%%
\bibliographystyle{abbrv}
\bibliography{multidimensional-osp}
\end{document}